\documentclass[12pt]{amsart}

\usepackage[all]{xypic}
\usepackage{tikz}
\usepackage{ragged2e}
\usetikzlibrary{arrows} 
\usetikzlibrary{decorations.markings}
\usepackage{graphicx}
\usepackage{bm}
\usepackage{float}
\usepackage{epsf}
\usepackage{verbatim} 
\usepackage{amsmath}
\usepackage{amsfonts}
\usepackage{amssymb}
\usepackage{mathrsfs}
\usepackage{amsthm}
\usepackage{newlfont}
\usepackage{enumitem}
\usepackage[new]{old-arrows}
\usepackage{booktabs}
\usepackage{enumitem}
\usepackage{makecell}
\usepackage[multiple]{footmisc}
\usepackage[answerdelayed,lastexercise]{exercise}
\usepackage[hidelinks]{hyperref}
\usepackage{mleftright}

\usepackage{fnpct}
\usepackage{colonequals} 
\usepackage{stmaryrd}

\usepackage[letterpaper, margin=1in]{geometry}

\usepackage{apptools}
\usepackage{chngcntr}

\newtheorem{thm}{Theorem}[section]
\newtheorem{prop}[thm]{Proposition}
\newtheorem{lem}[thm]{Lemma}
\newtheorem{cor}[thm]{Corollary}

\theoremstyle{definition}
\newtheorem{defn}[thm]{Definition}

\theoremstyle{remark}
\newtheorem{rem}[thm]{Remark}

\numberwithin{equation}{section}
\numberwithin{thm}{section}

\makeatletter 
\newcommand\mynobreakpar{\vspace{0.02in}\par\nobreak\@afterheading}  
\makeatother

\makeatletter
\@addtoreset{footnote}{section}
\makeatother

\usepackage{xcolor}
\hypersetup{
	colorlinks,
	linkcolor={red!50!black},
	citecolor={blue!50!black},
	urlcolor={blue!80!black}
}

\AtBeginDocument{%
	\def\MR#1{}
}

\DeclareMathOperator{\Hom}{Hom}
\DeclareMathOperator{\End}{End}
\DeclareMathOperator{\Aut}{Aut}
\DeclareMathOperator{\rank}{rank}

\DeclareMathOperator{\Frob}{Frob}

\DeclareMathOperator{\Mat}{Mat}

\DeclareMathOperator{\Gal}{Gal}

\DeclareMathOperator{\Ann}{Ann}

\DeclareMathOperator{\Cent}{Cent}

\newcommand{\id}{\mathrm{id}}

\newcommand{\ev}{\mathrm{ev}}

\newcommand{\fp}{\mathfrak{p}}

\newcommand{\F}{\mathbb{F}}

\newcommand{\To}{\longrightarrow}

\newcommand{\oF}{\overline{\F}}

\newcommand{\Mod}[1]{\ (\mathrm{mod}\ #1)}

\providecommand{\Aut}{\operatorname{Aut}}
\providecommand{\Hom}{\operatorname{Hom}}
\newcommand{\wtR}{\operatorname{wt}}
\newcommand{\dR}{d_R}
\newcommand{\kp}{\F_{\fp}}
\renewcommand{\ev}{\operatorname{ev}}

\newcommand{\ToFrom}[2]{\mathrel{\substack{%
\displaystyle\xrightarrow{\makebox[2.5em]{$\scriptstyle #1$}}\\[-0.4ex]
\displaystyle\xleftarrow[{\makebox[2.5em]{$\scriptstyle #2$}}]{}}}}

\title{Skew cyclic codes through Drinfeld modules}

\author{Giacomo Micheli}
\address{Department of Mathematics \& Statistics, The University of South Florida, Florida, United States of America}
\email{gmicheli@usf.edu}

\author{Mihran Papikian}
\address{Department of Mathematics, Pennsylvania State University, University Park, Pennsylvania, United States of America}
\email{papikian@psu.edu}

\thanks{The first author was supported by NSF CAREER grant 2338424. The second author was supported in part by the Simons Foundation, award number MPS-TSM-00008093.}

\subjclass[2020]{11G09, 11T71, 16S36, 94B15}
\keywords{Skew cyclic codes, rank-metric codes, MRD codes, Drinfeld modules, twisted polynomial rings}

\begin{document}

\begin{abstract}
We construct skew cyclic codes in the rank metric from the torsion of the
supersingular Drinfeld module $\phi_T=t+\tau^n$ over $\F_{q^r}$, where $t$
is a root of a monic irreducible $\fp_0\in\F_q[T]$ of degree $r$ coprime to
$n$. For every prime $\fp\neq\fp_0$ of degree $d$ and every divisor $m$ of
$n$, the torsion $\phi[\fp]$ is a free module over the subring
$R_0=\kappa[\pi^m]$ generated by a power of the Frobenius $\pi=\tau^r$,
$\kappa=\F_q[T]/\fp$, and the codes are the left ideals of the centralizer
$R\cong\Mat_m(R_0)$ of $R_0$. We prove a lattice
anti-isomorphism between codes and $R_0$-submodules of $\phi[\fp]$, and 
a rank-BCH bound whose designed distance is
the true distance. 
\end{abstract}

\maketitle

\section{Introduction}\label{s:intro}

Cyclic codes are the ideals of $\F_q[x]/(x^N-1)$, and their theory rests on
the correspondence between ideals and sets of roots of unity; cf. \cite[Ch.~6]{vanLint}. Rank-metric
codes, introduced by Delsarte \cite{Delsarte} and Gabidulin
\cite{Gabidulin}, replace the Hamming weight of a vector over $\F_{q^m}$ by
the $\F_q$-dimension of the span of its coordinates, and codes attaining the
corresponding Singleton bound are called maximum rank distance (MRD) codes.
The analogues of cyclic codes in this setting are the $q$-cyclic codes of
Gabidulin \cite{Gabidulin,Gabidulin2009} and, more generally, the skew
cyclic codes of Boucher, Geiselmann and Ulmer \cite{BGU}, Boucher and Ulmer
\cite{BoucherUlmer}, and Chaussade, Loidreau and Ulmer \cite{ChaussadeLU}:
left ideals in quotients $K[y;\theta]/(f)$ of twisted polynomial rings by
central elements, with the rank weight measured on the coefficient vector.
Chaussade, Loidreau and Ulmer proved a lower bound on the minimum rank
distance of such a code in terms of a Frobenius orbit contained in the root
space of its generator \cite[Prop.~1]{ChaussadeLU}, the rank-metric
counterpart of the BCH bound. Mart\'inez-Pe\~nas \cite{MP} then developed
the theory of $q^r$-cyclic codes over $\F_{q^m}$ systematically: root
spaces and cyclotomic spaces, a lattice anti-isomorphism between skew
cyclic codes and root spaces, rank-metric versions of the Hartmann--Tzeng
and shift bounds, and the identification of rank-BCH codes with the
generalized Gabidulin codes of \cite{KG} when the length equals $m$.

In \cite{MP26Rank,GMP}, rank-metric and sum-rank-metric codes were
constructed by restricting morphisms of Drinfeld modules to their torsion
submodules. The present paper applies the same principle to skew cyclic
codes and shows that the whole theory above lives on the torsion of a single
supersingular Drinfeld module. Let $A=\F_q[T]$, let $\fp_0\in A$ be monic
irreducible of degree $r$ with $\gcd(n,r)=1$, let $t\in\F_{q^r}$ be a root
of $\fp_0$, and let $\phi$ be the Drinfeld $A$-module of rank $n$ over
$\F_{q^r}$ given by
\[
\phi_T=t+\tau^n .
\]
The coprimality of $n$ and $r$ is exactly the condition that $\phi$ be
supersingular (Lemma~\ref{lem:height}), and then the Frobenius
$\pi\colonequals\tau^r$ of $\phi$ satisfies $\pi^n=\phi_{\fp_0}$. Fix a
prime $\fp\neq\fp_0$ of $A$ of degree $d$ and put $\kappa=A/\fp\cong\F_{q^d}$,
$M=\phi[\fp]\cong\kappa^n$, $E=\End_\kappa(M)$, and
$c=\fp_0\bmod\fp\in\kappa^\times$. Supersingularity makes the reduction map
$\End(\phi)\to E$ surjective, so every $\kappa$-linear operator on the
torsion arises from an endomorphism of $\phi$. For a divisor $m\mid n$ we consider
the commutative subring $R_0=\kappa[\pi^m|_M]$ of $E$ and its centralizer
\[
R\colonequals\Cent_E(R_0). 
\]
We define a \emph{skew cyclic code} to be a left ideal of $R$. The rank
weight comes from the constants: $L=\F_{q^m}\subseteq\F_{q^n}$ acts on $M$
by scalar multiplication, and $\wtR(f)$ is the $\F_q$-dimension of the space
of coefficients in $L$ needed to write $f$ over the ring $S=\kappa[\pi|_M]$.

The first group of results describes this structure
(Section~\ref{s:module}). \emph{The Frobenius $\pi$ is non-derogatory on
$\phi[\fp]$}: its minimal polynomial over $\kappa$ is $X^n-c$, even when it is inseparable
(Proposition~\ref{prop:nondero}). The proof is a saturation argument in
$\End(\phi)$ rather than an appeal to semisimplicity, and some argument is
genuinely needed, since minimal polynomials do not commute with reduction in
general. Consequently \emph{$\phi[\fp]$ is free of rank one over $S$ and
free of rank $m$ over $R_0\cong\kappa[u]/(u^s-c)$, $s=n/m$}
(Proposition~\ref{prop:free}); this is a normal basis theorem for the torsion of a
supersingular Drinfeld module. We then show that $R=\End_{R_0}(M)\cong\Mat_m(R_0)$,
$\Cent_E(R)=R_0$, 
and multiplication induces a bijection $L\otimes_{\F_q}S\to R$ 
(Proposition~\ref{prop:dc}). These are statements about supersingular
Drinfeld modules of independent interest: they exhibit, inside
$\End(\phi)/\fp\End(\phi)\cong\Mat_n(\kappa)$, an explicit
double-centralizer pair for each divisor of the rank, and they describe the
torsion as a cyclic module over the subalgebra generated by the Frobenius.

The second group of results concerns the codes. When $\gcd(m,d)=1$, the ring $R$ is a
quotient of a twisted polynomial ring over a field,
\[
R\cong K[y;\sigma]/(y^n-c),\qquad K\cong\F_{q^{md}},
\]
with $\sigma$ of order $m$ and fixed field $\kappa$
(Theorem~\ref{thm:presentation}). Thus, our codes are precisely the
$\theta$-codes of \cite{ChaussadeLU} with a central binomial modulus, and
the classical theory of generator and check polynomials is available
(Theorem~\ref{thm:generator}). Independently of any coprimality,
\emph{root spaces} are defined intrinsically as $R_0$-submodules of $M$, and
the maps $C\mapsto Z(C)$ (common kernel) and $W\mapsto\Ann(W)$ are mutually
inverse anti-isomorphisms between the lattice of codes and the lattice of
root spaces (Theorem~\ref{thm:lattice}, Corollary~\ref{cor:lattice}). The
proof is a duality argument for the Gorenstein algebra $R_0$ and uses
neither the Euclidean structure nor any separability; the problem of
recognizing root spaces, solved in \cite[Cor.~3]{MP} by a Galois-closure
criterion, disappears, because stability under $\pi^m$ is the definition.
Section~\ref{s:bounds} proves the rank-BCH bound: if $\alpha$ is $S$-normal
and $f\neq0$ vanishes on $\alpha,\alpha^{[r]},\dots,\alpha^{[r(\delta-2)]}$,
then $\wtR(f)\geq\delta$ (Theorem~\ref{thm:rankbch}). The Moore-matrix input
is a restricted nondegeneracy statement (Theorem~\ref{thm:moorefamily}),
proved by a kernel-rationality lemma and a Galois descent; it replaces
\cite[Lem.~2]{KG} and continues to function at primes where $M$ is not a
field. For the rank-BCH code $C=\Ann(W_\delta)$,
$W_\delta=\bigoplus_{i\leq\delta-2}R_0\alpha^{[ri]}$, we prove that
\emph{the designed distance is the true distance}, $\dR(C)=\delta$, with
$\dim_{\F_q}C=nd(m-\delta+1)$ (Theorem~\ref{thm:exact}).
Section~\ref{s:matrix} turns this into matrices: the coefficient map
$f\mapsto A_f\in\Mat_{m\times nd}(\F_q)$ is a rank isometry, and
\emph{every rank-BCH code is an $\F_{q^m}$-linear MRD code in
$\Mat_{m\times nd}(\F_q)$}, for all $m\mid n$, all $\fp\neq\fp_0$ and all
$1\leq\delta\leq m$ (Theorem~\ref{thm:mrd}); at $d=1$, $m=n$ these are the
generalized Gabidulin codes of \cite{KG}. The matrix codes carry a skew
constacyclic symmetry $A\mapsto QAT$ (Proposition~\ref{prop:shift}).
Finally, by Dirichlet's theorem for $\F_q[T]$, used for a similar purpose
in \cite{BDM24}, the construction can be run backwards from prescribed
parameters: for every $K=\F_{q^{md}}$ with
$\gcd(m,d)=1$, every automorphism $\sigma$ of $K$ of order $m$, and every
nonzero $c$ in the fixed field of $\sigma$, the skew constacyclic setting
$K[y;\sigma]/(y^n-c)$ arises from a supersingular Drinfeld module
(Proposition~\ref{prop:dirichlet}).

What is new here for coding theory, in relation to \cite{ChaussadeLU} and
\cite{MP}, can be summarized as follows.
\begin{enumerate}
\item The skew cyclic codes of \cite{MP} occupy a single prime of our
construction: at $\fp=\fp_0-1$, when this is irreducible, the ring $R$ is
the base change from $\F_q$ to $\F_{q^r}$ of the ring of $q^r$-polynomials
over $\F_{q^m}$ modulo $x^{[rn]}-x$ of \cite{MP}, and at $r=1$ the two
theories coincide exactly (Lemma~\ref{lem:mp}). For a general prime the
modulus is instead $y^n-c$, and the torsion module on which the codes act
need not be a field. Thus the construction leaves the $q^r$-cyclic
framework of \cite{MP}. When $\gcd(m,d)=1$, it still lies in the more
general $\theta$-code framework of \cite{ChaussadeLU}. When
$\gcd(m,d)>1$, the ring $K$ is not a field, and the codes leave that
framework as well.
\item Both \cite{ChaussadeLU} and \cite{MP} prove lower bounds on the
minimum rank distance, and in \cite{MP} the rank-BCH family is shown to be
MRD in the case $n=m$ (with $\gcd(r,n)=1$). For the present family we
prove the exact equality $\dR(C)=\delta$ and the MRD property for all
parameters covered by the construction; neither appears in
\cite{ChaussadeLU,MP}.
\item The rank-BCH bound \cite[Cor.~6]{MP}, the rank-Hartmann--Tzeng bound
\cite[Cor.~5]{MP}, and \cite[Prop.~7]{MP} are false as stated: they assume
the linear independence of the first $\delta-1$ Frobenius twists of
$\alpha$ only, whereas the argument requires the full orbit, which is the
hypothesis of \cite[Prop.~1]{ChaussadeLU}. We give a counterexample and
locate the gap in Section~\ref{sComparison}.
\item The lattice anti-isomorphism does not require the minimal
polynomial of $\pi^m|_M$ to be separable, so the repeated-root case
$p\mid n$ is included.
\item The Dirichlet argument shows that the Drinfeld-module construction is
universal for central binomial moduli.
\end{enumerate}
Our rank weight is that of \cite{MP}, measured over $\F_q$ on coefficients
in $\F_{q^m}$; the weight of \cite{ChaussadeLU} is measured over the field
of constants $\kappa$, and the two coincide when $d=1$. Section~\ref{ss:axioms} isolates the three properties of the pair
$(R,R_0)$ that the proofs consume, in the hope that they admit instances
beyond the family treated here.

\section{The module $\phi_T=t+\tau^n$}\label{s:module}

Throughout, $p$ is a prime, $q$ a power of $p$, $A=\F_q[T]$,
$F=\F_q(T)$, and $\tau$ denotes the $q$-power Frobenius $x\mapsto x^q$.
For an integer $i\geq0$, we write $[i]=q^i$, so $\beta^{[i]}=\beta^{q^i}$.
Twisted polynomial rings $K\{\tau\}$ (with $\tau c=c^q\tau$ for $c\in K$)
and Drinfeld modules are used as in \cite{PapikianBook}.

Fix $n\geq1$, a monic irreducible $\fp_0\in A$ of degree $r$, and a root
$t\in\F_{q^r}$ of $\fp_0$. Let $\phi$ be the Drinfeld module of rank
$n$ over the $A$-field $\F_{q^r}$ (with $\gamma\colon A\to\F_{q^r}$,
$T\mapsto t$) defined by
\[
\phi_T=t+\tau^n .
\]
The \textit{$A$-characteristic} of $\phi$ is $\fp_0$. Denote $\F_{\fp_0}=A/\fp_0\cong \F_{q^r}$. 

\begin{rem}
	A warning on notation: throughout, $n$ is the rank of $\phi$,
	$r=\deg\fp_0$ the degree of the $A$-characteristic, and $d=\deg\fp$ the
	degree of the code prime. In \cite{PapikianBook} and 
	\cite{MicheliPapikian},  $r$ denotes the rank and $d$ the degree of the characteristic. 
	We change the notation here to be more consistent with the notation in \cite{MP}. 
\end{rem}

Set $g\colonequals\gcd(n,r)$ and
$\pi\colonequals\tau^r$. Recall that the height $H(\phi)\in\{1,\dots,n\}$
satisfies \[\phi[\fp_0]\cong\F_{\fp_0}^{\, n-H(\phi)},\] and that
$\phi$ is \textit{supersingular} if and only if $H(\phi)=n$
\cite[Thm. 4.4.1]{PapikianBook}.

\begin{lem}\label{lem:height}
$H(\phi)=n/g$. In particular $\phi$ is supersingular if and only if
$g=1$, and in that case
\[
\phi_{\fp_0}=\tau^{nr}=\pi^n .
\]
\end{lem}

\begin{proof}
See \cite[\S6.1]{MicheliPapikian}. 
\end{proof}

Assume from now on that $g=\gcd(n,r)=1$, and enlarge the base field to
$\F_{q^{nr}}$. Since $\phi$ is defined over $\F_{q^r}$, the element
$\pi=\tau^r$ is an endomorphism of $\phi$, called the \textit{Frobenius} of $\phi$
over $\F_{q^r}$. By  Lemma~\ref{lem:height}, $\pi^n=\phi_{\fp_0}$, so the minimal polynomial of $\pi$
over $F$ divides $X^n-\fp_0$, which is irreducible (Eisenstein at
$\fp_0$). Having degree $n=\operatorname{rank}\phi$, it is also the
characteristic polynomial of the Frobenius:
\begin{equation}\label{eq:charpoly}
P_\pi(X)=X^n-\fp_0 .
\end{equation}

\begin{lem}\label{lem:endfamily}
$\End(\phi)=\F_{q^n}[\phi_T,\pi]$, and all endomorphisms of $\phi$ are
defined over $\F_{q^{nr}}$. As an $A$-module, $\End(\phi)$ is free with
basis $\{\zeta_a\pi^b:0\leq a,b\leq n-1\}$, where
$\zeta_0,\dots,\zeta_{n-1}$ is an $\F_q$-basis of the constants
$\F_{q^n}\subseteq\End(\phi)$. The relations are $\pi^n=\phi_{\fp_0}$
and $\pi\zeta=\zeta^{[r]}\pi$, $\zeta\in\F_{q^n}$. Since
$\gcd(n,r)=1$, conjugation by $\pi$ generates $\Gal(\F_{q^n}/\F_q)$,
and $\End(\phi)$ is a maximal order in the cyclic division algebra
$\End(\phi)\otimes_AF$ of dimension $n^2$ over $F$.
\end{lem}

\begin{proof}
See \cite[\S6.1]{MicheliPapikian}. 
\end{proof}

It will be convenient to write
\[
B\colonequals\F_{q^n}[\phi_T]=\bigoplus_{i=0}^{n-1}A\,\zeta_i
\subseteq\End(\phi),
\]
so that Lemma~\ref{lem:endfamily} reads
\begin{equation}\label{eq:graded}
\End(\phi)=\bigoplus_{j=0}^{n-1}B\,\pi^j.
\end{equation}
Identifying $\phi_T$ with $T$, the ring $B$ is the polynomial ring
$\F_{q^n}[T]$, the integral closure of $A$ in the constant field
extension $\F_{q^n}F$. In particular, 
\begin{equation}\label{eq:olcapf}
B\cap F=A. 
\end{equation}

\vspace{0.1in}

Fix a monic irreducible $\fp\in A$ with $\fp\neq\fp_0$, put
$d=\deg\fp$, and let
\[
\kappa\colonequals\kp=A/\fp\cong\F_{q^d},\qquad
M\colonequals\phi[\fp]\subseteq\oF_{\fp_0}.
\]
Since $\fp$ differs from the $A$-characteristic, $M\cong\kappa^{\,n}$
is a free $\kappa$-module of rank $n$; see \cite[Thm. 3.5.2]{PapikianBook}.  Put
\[
E\colonequals\End_{\kappa}(M)\cong\Mat_n(\kappa),\qquad
c\colonequals\fp_0\bmod\fp\ \in\ \kappa^\times .
\]

The reduction map $\End(\phi)\to E$ factors through
$\End(\phi)/\fp\End(\phi)$ and is injective on the quotient; see \cite[Lem. 3.4.10]{PapikianBook}. 
Supersingularity is exactly the condition that the induced injection be
an isomorphism: both $\End(\phi)/\fp\End(\phi)$ and $E$ are free of
rank $n^2$ over $\kappa$ \cite[$\S$~4.4]{PapikianBook}, so every
$\kappa$-linear operator on the torsion comes from an endomorphism of
$\phi$.

\begin{prop}\label{prop:nondero} We have: 
\begin{enumerate}
\item $\pi^n$ acts on $M$ as the scalar $c$, and the characteristic
polynomial of $\pi|_M$ over $\kappa$ is $X^n-c$, i.e., the reduction of
\eqref{eq:charpoly} modulo $\fp$.
\item The minimal polynomial of $\pi|_M$ over $\kappa$ is also
$X^n-c$; that is, $\pi|_M$ is non-derogatory. 
\end{enumerate}
\end{prop}

\begin{proof}
(1) $\pi^n=\phi_{\fp_0}$ acts on the $A/\fp$-module $M$ as
multiplication by $\fp_0\bmod\fp=c$. The characteristic polynomial of
the Frobenius acting on the $\fp$-torsion is the reduction modulo
$\fp$ of its characteristic polynomial \eqref{eq:charpoly} over $A$; cf. 
\cite[Ch.~4]{PapikianBook}.

(2) Let $\bar\mu\in\kappa[X]$ be monic of degree $<n$ with
$\bar\mu(\pi|_M)=0$, and lift it to a monic
$\mu=\sum_{j<n}\mu_j X^j\in A[X]$ of the same degree. Since $a\in A$
acts on $M$ through $a\bmod\fp$, the endomorphism
\[
\mu(\pi)=\sum_{j<n}\phi_{\mu_j}\,\pi^j\ \in\ \End(\phi)
\]
vanishes on $M$, hence lies in $\fp\End(\phi)$ by \cite[Lem. 3.4.10]{PapikianBook}.
By \eqref{eq:graded}, $\fp\End(\phi)=\bigoplus_{j=0}^{n-1}\fp B\,\pi^j$, and
comparing components, $\mu_j\in\fp B\cap A$ for every $j$ (here we use the assumption that $\deg \mu \leq n-1$).  By
\eqref{eq:olcapf}, $\mu_j/\fp\in B\cap F=A$, i.e.\
$\mu_j\in\fp A$ for every $j$. This contradicts $\mu$ being monic.
Hence the minimal polynomial of $\pi|_M$ has degree $n$. Since
$\pi|_M$ satisfies $X^n-c$ by (1), the two coincide.
\end{proof}

Fix a divisor $m\mid n$ and set $s=n/m$. Two commutative subrings of
$E$ will play an important role throughout the paper:
\[
S\colonequals\kappa[\pi|_M],\qquad
R_0\colonequals\kappa[\pi^m|_M].
\]

\begin{prop}\label{prop:free}
$M$ is free of rank $m$ over $R_0$. More precisely:
\begin{enumerate}
\item $S\cong\kappa[x]/(x^n-c)$, and $M$ is free of rank one over $S$. 
\item $S=\bigoplus_{i=0}^{m-1}\pi^iR_0$ is free of rank $m$ over
$R_0$, and $R_0\cong\kappa[u]/(u^s-c)$ with $u=\pi^m|_M$. Hence
$M\cong S\cong R_0^{\oplus m}$.
\end{enumerate}
\end{prop}

\begin{proof}
(1) By Proposition~\ref{prop:nondero}(2) the minimal polynomial of
$\pi|_M$ is $X^n-c$, so $S\cong\kappa[x]/(x^n-c)$ has
$\kappa$-dimension $n=\dim_\kappa M$. A finite-dimensional module over
$\kappa[x]$ whose annihilator has degree equal to its dimension is
cyclic. Thus, \[M\cong\kappa[x]/(x^n-c)\cong S,\] is free of rank one over $S$. 

(2) Note that $\kappa[x]$ is a free $\kappa[x^m]$-module of rank $m$: 
\begin{equation}\label{eq:directsum2}\kappa[x]=\bigoplus_{i=0}^{m-1}x^i\,\kappa[x^m].
\end{equation}
The ideal $(x^n-c)\lhd \kappa[x]$ is
\textit{compatible} with this grading, 
\begin{equation}\label{eq:directsum}
(x^n-c)=\bigoplus_{i=0}^{m-1}x^i\,(x^n-c)\kappa[x^m],
\end{equation}
because the generator $x^n-c=(x^m)^s-c$ lies in the $0$-th component 
(here we use the assumption that $m\mid n$).  
The decomposition \eqref{eq:directsum}  implies that 
\[(x^n-c)\cap\kappa[x^m]=(x^n-c)\kappa[x^m].\] 
Since $R_0$ is the
image of $\kappa[x^m]$ in $S$, we get 
\[R_0\cong \kappa[x^m]/(x^n-c)\kappa[x^m]\cong\kappa[u]/(u^s-c).\] 
Next, quotienting the direct sum \eqref{eq:directsum2} by a grading compatible ideal \eqref{eq:directsum} gives 
a direct sum 
\[S=\bigoplus_{i<m}\pi^iR_0,\]  
which is the freeness of $S$ over $R_0$ with basis $1,\pi,\dots,\pi^{m-1}$.
Finally, combining with (1), 
\[M\cong S\cong R_0^{\oplus m}\]
is a free $R_0$-module of rank $m$. 
\end{proof}

\begin{rem}\label{rem1.6} 
Proposition~\ref{prop:free}(1) can be thought of as the Normal Basis Theorem for the
torsion module $\phi[\fp]$. Even more directly, if $\fp_0-1$ is irreducible and we take $\fp=\fp_0-1$,  then 
$\phi_{\fp_0-1}=\tau^{nr}-1$, so
$M=\F_{q^{nr}}$. In this case, $\pi|_M$ is the $q^r$-power Frobenius of
$\F_{q^{nr}}/\F_{q^r}$, so Proposition~\ref{prop:free}(1) is a
counterpart of the normal basis theorem for this extension. (The two
statements coincide literally when $r=1$; for $r\geq2$, note that
$\kappa$ acts on $M$ through $\phi$, not by scalar multiplication.) 
\end{rem}

Put
\[
R\colonequals\Cent_E(R_0),\qquad L\colonequals\F_{q^m},
\]
the latter viewed inside $E$ as scalar multiplications by the
constants \[\F_{q^m}\subseteq\F_{q^n}\subseteq\End(\phi).\]

\begin{prop}\label{prop:dc}
With previous notation:
\begin{enumerate}
\item $R=\End_{R_0}(M)\cong\Mat_m(R_0)$.
\item $\Cent_E(R)=R_0$.
\item $R$ is the image of
$B_m\colonequals\F_{q^m}[\phi_T,\pi]\subseteq\End(\phi)$ in $E$, and
$L$ is precisely the constant part of $R$.
\item Multiplication
\[
L\otimes_{\F_q}S\ \To\ R
\]
is bijective.
\end{enumerate}
\end{prop}

\begin{proof}
(1) First, observe that $\kappa\subseteq R_0$ by definition.
Consequently an $R_0$-linear endomorphism of $M$ is automatically
$\kappa$-linear, and
\[
R=\Cent_E(R_0)=\End_{R_0}(M)
\]
is a tautology: centralizing $R_0$ inside the $\kappa$-linear maps is
the same as being $R_0$-linear. By Proposition~\ref{prop:free},
$M\cong R_0^{\oplus m}$, so $R\cong \End_{R_0}(M)\cong\Mat_m(R_0)$. Note that this implies 
\[\dim_{\F_q}R=m^2\cdot sd=mnd.\]

(2) Suppose $f\in E$ commutes with $R$. Since $R_0$ sits inside $R$ as
the scalar matrices, $f$ commutes with $R_0$, hence
$f\in\End_{R_0}(M)=R$ by (1). An element of $\Mat_m(R_0)$ central in
$\Mat_m(R_0)$ lies in $Z\bigl(\Mat_m(R_0)\bigr)=R_0$, as $R_0$ is
commutative. Thus, $\Cent_E(R)=R_0$.

(3) Choose the basis $\zeta_0,\dots,\zeta_{n-1}$ of
Lemma~\ref{lem:endfamily} so that $\zeta_0,\dots,\zeta_{m-1}$ is a
basis of $\F_{q^m}$ (possible as $m\mid n$). One checks from the
relations that
\[
B_m=\bigoplus_{i<m,\ j<n}A\,\zeta_i\pi^j. 
\]
Thus, $B_m$ is an $A$-module direct summand of $\End(\phi)$, so 
its image in $E=\End(\phi)/\fp\End(\phi)$ has $\F_q$-dimension
$mnd$. The image lies in $R$: the image of $\phi_T$ is a scalar in
$\kappa\subseteq R_0$; $\pi$ commutes with $R_0=\kappa[\pi^m]$; and
$\zeta\in\F_{q^m}$ commutes with $\pi^m$ (as $\zeta^{[rm]}=\zeta$) and
with the image of $A$ (as $\zeta^{[n]}=\zeta$). Comparing dimensions,
the image equals $R$. As for the constants, $\zeta\in\F_{q^n}$
centralizes $R_0$ if and only if it commutes with $\pi^m$, i.e.,
$\zeta^{[rm]}=\zeta$. Thus, 
\[\zeta\in\F_{q^n}\cap\F_{q^{rm}}=\F_{q^{\gcd(n,rm)}}=\F_{q^m},\] using
$\gcd(n,r)=1$ and $m\mid n$.

(4) The image of $B_m$ is spanned by the products
$\zeta\,\phi_a\,\pi^j$ with $\zeta\in L$, $a\in A$, i.e., equals $L\cdot S$, so
multiplication $L\otimes_{\F_q}S\to R$ is surjective. Since
\[\dim_{\F_q}(L\otimes_{\F_q}S)=m\cdot nd=\dim_{\F_q}R,\] it is
bijective.
\end{proof}

\section{Skew cyclic codes}\label{s:codes}

\begin{defn} With notation of the previous section, 
a \textit{skew cyclic code} (at the prime $\fp$, with parameter
$m\mid n$) is a left ideal $C\subseteq R$.
\end{defn}

\begin{prop}\label{prop:principal}
For every $m\mid n$ and every $\fp\neq\fp_0$, each left ideal of
$R\cong\Mat_m(R_0)$ is principal.
\end{prop}

\begin{proof}
Identify $M\cong R_0^{\oplus m}$ (columns) and $R\cong\Mat_m(R_0)$.
For a left ideal $C$, let $U(C)\subseteq R_0^{\,m}$ (rows) be the set
of all rows of all elements of $C$. From 
$(\alpha e_{ij})X\in C$ for all $X\in C$ and $\alpha\in R_0$, one sees that $U(C)$ is an $R_0$-submodule, and that
\[
C=\{X\in R:\ \text{every row of }X\text{ lies in }U(C)\}.
\]
Now $R_0\cong\kappa[u]/(u^s-c)$ is a quotient of the principal ideal
domain $\kappa[u]$, so $U(C)$, being the image of a submodule of
$\kappa[u]^{\,m}$ containing $(u^s-c)\kappa[u]^{\,m}$, is generated by
at most $m$ elements. Let $G\in R$ be a matrix whose rows are a set of
$m$ generators of $U(C)$ (repeating a generator if necessary). Then $RG$
consists exactly of the matrices whose rows lie in the $R_0$-span of
the rows of $G$, i.e.,\ $RG=C$.
\end{proof}

Let
\[
K\colonequals\text{the image of }
L\otimes_{\F_q}\kappa\ \text{ in }R
\]
under multiplication; by Proposition~\ref{prop:dc}(4) (restrict the
splitting to $\kappa\subseteq S$) this map is injective, so
$K\cong L\otimes_{\F_q}\kappa$ as rings, a commutative subring of $R$
containing $L$ and the central $\kappa$.

\begin{thm}\label{thm:presentation}
Assume $\gcd(m,d)=1$. Then:
\begin{enumerate}
\item $K\cong\F_{q^{md}}$ is a field.
\item Conjugation by $\pi$ preserves $K$ and induces the automorphism
$\sigma\in\Aut(K)$ acting as the identity on $\kappa$ and as
$\zeta\mapsto\zeta^{[r]}$ on $L$. Explicitly, $\sigma=\Frob^e|_K$, where
$\Frob\colon x\mapsto x^q$ denotes the $q$-power Frobenius automorphism,
and $e$ is determined by $e\equiv r\Mod m$, $e\equiv0\Mod d$. The
automorphism $\sigma$ has order $m$ and fixed field $\kappa$.
\item Let $K[y;\sigma]$ be the twisted polynomial ring with
$yx=\sigma(x)y$ for $x\in K$. The assignment $y\mapsto\pi$ induces an
isomorphism
\[
K[y;\sigma]/(y^n-c)\ \xrightarrow{\ \sim\ }\ R ,
\qquad R=\bigoplus_{i=0}^{n-1}K\pi^i ,
\]
and the modulus $y^n-c$ is central in $K[y;\sigma]$.
\end{enumerate}
\end{thm}

\begin{proof}
(1) A tensor product of finite fields of coprime degrees over $\F_q$
is a field: \[\F_{q^m}\otimes_{\F_q}\F_{q^d}\cong\F_{q^{md}}  \text{ when }
\gcd(m,d)=1.\]

(2) The operator $\pi|_M$ is invertible ($\pi^n=c\in\kappa^\times$).
For $\zeta\in L$ (a scalar) and $v\in M$ one has
$\pi(\zeta v)=\zeta^{[r]}\pi(v)$, so $\pi\zeta\pi^{-1}=\zeta^{[r]}$. Next, 
$\pi$ commutes with $\kappa$, because being an endomorphism of $\phi$,
$\pi$ commutes with the action of $A$ on $\phi[\fp]$, through which
$\kappa$ acts, i.e.,  $\pi\in E$ and $\kappa=Z(E)$.
Hence conjugation by
$\pi$ preserves $K=L\cdot\kappa$ and acts as described. On
$K\cong\F_{q^{md}}$ every automorphism is a power of $\Frob$, and the
stated congruences (solvable by the Chinese remainder theorem, as
$\gcd(m,d)=1$) characterize the restriction to the two factors. Since
$m\mid n$ and $\gcd(n,r)=1$, we have $\gcd(m,r)=1$, so
$\gcd(e,m)=1$ and $\gcd(e,md)=d$. Therefore, $\sigma$ has
order $md/d=m$ and fixed field $\F_{q^d}=\kappa$.

(3) Regrouping the splitting of Proposition~\ref{prop:dc}(4) and using
$S=\bigoplus_{i<n}\kappa\pi^i$ (Proposition~\ref{prop:free}(1)),
\[
R\cong L\otimes_{\F_q}S
=(L\otimes_{\F_q}\kappa)\otimes_{\kappa}S
=\bigoplus_{i=0}^{n-1}K\pi^i .
\]
The assignment $y\mapsto\pi$ respects the defining relation
($\pi x=\sigma(x)\pi$ for $x\in K$ by (2)) and kills $y^n-c$, so it
induces a surjection $K[y;\sigma]/(y^n-c)\to R$ by the displayed
decomposition. Both sides have $\F_q$-dimension $md\cdot n=mnd$, so it
is an isomorphism. Finally, $y^n$ commutes with $K$
because $\sigma^n=\id$ ($m\mid n$), and $c\in\kappa$ is fixed by
$\sigma$; thus, $yc=cy$ and $y^n-c$ is central.
\end{proof}

\begin{rem}\label{rem:runiform}
	For $\fp_0=T$  (so $\phi_T=\tau^n$ over $\F_q$), the
	ring $R$ admits a second, finer presentation
	\[
	R\ \cong\ \F_{q^m}\{\tau\}/(\phi_\fp),
	\]
	valid for every $\fp$ and every $m\mid n$, with no condition on
	$\gcd(m,d)$: the image of $\F_{q^m}\{\tau\}$ in $E$ centralizes
	$R_0=\kappa[\tau^m]$, its kernel is $(\phi_\fp)$, and the dimensions
	agree. In this presentation the codes are the pseudo-$q$-cyclic codes of
	\cite{Gabidulin2009} and \cite[Rem.~1]{MP}, with central modulus
	$\phi_\fp$; see Section~\ref{s:matrix}.
\end{rem}

\begin{thm}\label{thm:generator}
Assume $\gcd(m,d)=1$ and identify $R=K[y;\sigma]/(y^n-c)$ as in
Theorem~\textup{\ref{thm:presentation}}. Let $C\subseteq R$ be a
nonzero left ideal. There is a unique monic $g\in K[y;\sigma]$ of
minimal $y$-degree whose class lies in $C$, and $C=Rg$. There is a
unique $h\in K[y;\sigma]$ with $y^n-c=hg=gh$. Moreover:
\begin{enumerate}
\item $f\in C$ if and only if $g$ divides $f$ on the right in
$K[y;\sigma]$. 
\item $\dim_{K}C=n-\deg_y g$, with basis $g,yg,\dots,y^{k-1}g$, where
$k=n-\deg_y g$. 
\item $f\in C$ if and only if $fh=0$ in $R$.
\end{enumerate}
\end{thm}

\begin{proof}
The ring $K[y;\sigma]$ is left and right Euclidean with respect to
$\deg_y$ (see \cite{Ore}), and the arguments of
\cite[Cor.~3.1.15, 3.1.16]{PapikianBook}, stated there for twisted
polynomial rings with the $q$-power twist, apply verbatim to any
twisted polynomial ring over a field. Thus the preimage of $C$ in
$K[y;\sigma]$ is a principal left ideal containing the central
$y^n-c$, generated by the unique monic $g$ of minimal $y$-degree in
it, which gives the first claim and (1). Since $y^n-c$ is central
with right divisor $g$, there is a unique $h$ with $y^n-c=hg=gh$.

(2) The classes of $1,y,\dots,y^{n-1}$ form a $K$-basis of $R$. The
classes $y^ig$, $0\leq i\leq k-1$, have distinct $y$-degrees, hence are linearly independent over $K$. 
They span $C$, because any $f\in C$ equals $ug$ with
$\deg_y u<k$ by (1), so is a $K$-linear combination of the $y^ig$.

(3) If $f=ug$, then $fh=u(gh)=u(y^n-c)=0$ in $R$. Conversely, suppose
$fh=0$ and write $f=ug+v$ with $\deg_y v<\deg_y g$ by left division.
Then $vh\equiv0\pmod{y^n-c}$ while
$\deg_y(vh)<\deg_y g+\deg_y h=n$, so $vh=0$, and $v=0$ by
cancellation in the domain $K[y;\sigma]$. 
\end{proof}

\begin{rem}\label{rem:constacyclic}
	Under the coordinate map $K^{n}\to R$,
	\[(c_0,\dots,c_{n-1})\mapsto\sum c_iy^i,\] left multiplication by $y$ is
	the $\sigma$-twisted \emph{constacyclic} shift of length $n$ with
	feedback constant $c$:
	\[
	(c_0,\dots,c_{n-1})\longmapsto
	\bigl(c\,\sigma(c_{n-1}),\sigma(c_0),\dots,\sigma(c_{n-2})\bigr).
	\]
	These are exactly the \textit{$\theta$-codes} of \cite{ChaussadeLU} with
	central modulus $y^n-c$, over the alphabet $K\cong\F_{q^{md}}$, with
	$\theta=\sigma$ of order $m$. 
\end{rem}

\section{Root spaces and the lattice anti-isomorphism}\label{s:roots}
\begin{defn}
Motivated by the root spaces of \cite{MP}, we call an $R_0$-submodule
of $M$ a \emph{root space}. Equivalently, it is a
$\kappa$-subspace of $M$ stable under $\pi^m$.
\end{defn}

For a code (i.e., left ideal)
$C\subseteq R$ and a root space $W\subseteq M$, set
\begin{align*}
	Z(C)&=\{v\in M:\ f(v)=0\ \text{for all }f\in C\},\\
	\Ann(W)&=\{f\in R:\ f(W)=0\}.
\end{align*}
Observe that $Z(C)$ is an $R_0$-submodule of $M$ because
$R$ centralizes $R_0$. Also, $\Ann(W)$ is a left ideal for any subset
$W$. Thus, we get the maps
\[
\{\text{Left ideals of $R$}\}\ \ToFrom{Z}{\Ann}\
\{\text{$R_0$-submodules of $M$}\}
\]
\begin{thm}\label{thm:lattice} The above maps are
	mutually inverse bijections.
\end{thm}
\begin{proof}
	Let $C\subseteq R$ be a left ideal. Let $U_C\subseteq R_0^m$ (rows) be the set of all rows of all elements of $C$.
	As we explained in the proof of Proposition~\ref{prop:principal},
	$U_C$ is an $R_0$-submodule of $R_0^m$. Now
	\begin{align*}
		Z(C) &= \{x\in M:\ U_C x = 0\}, \\
		U_{\Ann(W)} &= \{u\in R_0^m:\ u W = 0\},
	\end{align*}
	where the notation indicates the row-column pairing being identically zero.
	Now the claim of the theorem becomes the assertion that the pairing
	\[
	R_0^m\times R_0^m\To R_0, \qquad (u, x)\longmapsto \sum u_ix_i
	\]
	induces mutually inverse, inclusion reversing bijections between the $R_0$-submodules
	of row and column spaces. This holds because $R_0=\kappa[u]/(u^s-c)$
	is a \textit{Frobenius}, or \textit{Gorenstein}, algebra; cf.
	\cite[\S16]{Lam}. Explicitly, let $\lambda\colon R_0\To\kappa$ pick
	the coefficient of $u^{s-1}$ in the basis $1,u,\dots,u^{s-1}$. The
	kernel of $\lambda$ contains no nonzero ideal, since for
	$0\neq x=\sum_{i\leq e}x_iu^i$ with $x_e\neq0$ one has
	$\lambda(u^{s-1-e}x)=x_e\neq0$. Equivalently, the bilinear form
	$(a,b)\mapsto\lambda(ab)$ on $R_0$ is nondegenerate, or again,
	$a\mapsto\lambda(\,\cdot\,a)$ is an isomorphism
	$R_0\cong\Hom_\kappa(R_0,\kappa)$ of $R_0$-modules. Consequently
	the $\kappa$-bilinear form
	\[
	B(u,x)=\lambda\Bigl(\sum_{i=1}^m u_ix_i\Bigr)
	\]
	on $R_0^m\times R_0^m$ is nondegenerate. For an $R_0$-submodule $W$
	of the column space and a row $u$, the set
	$uW=\{\sum_i u_ix_i:\ x\in W\}$ is an ideal of $R_0$, so
	$\lambda(uW)=0$ if and only if $uW=0$. Hence $\{u:\ uW=0\}$ is the
	orthogonal complement $W^\perp$ of $W$ with respect to $B$, and
	likewise $\{x:\ Ux=0\}=U^\perp$ for a submodule $U$ of the row
	space. The two maps are therefore double orthogonal complements with
	respect to a nondegenerate form on a finite-dimensional
	$\kappa$-space, and $W^{\perp\perp}=W$, $U^{\perp\perp}=U$.
\end{proof}
	
\begin{cor}\label{cor:lattice} The maps $\Ann$ and $Z$ have the following properties.
	\begin{enumerate}
		\item The maps $\Ann$ and $Z$ are anti-isomorphisms of lattices:
		\begin{align*}
			Z(C_1+C_2) &=Z(C_1)\cap Z(C_2),\\  Z(C_1\cap C_2) &=Z(C_1)+Z(C_2), \\ 
			\Ann(W_1+W_2) &=\Ann(W_1)\cap \Ann(W_2),\\  \Ann(W_1\cap W_2) &=\Ann(W_1)+\Ann(W_2). 
		\end{align*}
		\item For every root space $W$,
		\[
		\dim_\kappa\Ann(W)=m\,\bigl(n-\dim_\kappa W\bigr).
		\]
		If $\gcd(m,d)=1$ and $C=Rg$ with $g$ the minimal monic generator
		of Theorem~\textup{\ref{thm:generator}}, then 
		\[
		\dim_{\kappa}Z(C)=\deg_y g=n-\dim_{K}C .
		\]
	\end{enumerate}
\end{cor}
\begin{proof}
	(1) Both maps are inclusion-reversing, directly from the
	definitions, and for any code $C$ and root space $W$ one has 
	\[
	C\subseteq\Ann(W)\iff W\subseteq Z(C),
	\]
	since both containments state that $f(v)=0$ for all $f\in C$,
	$v\in W$. Hence, for every root space $W$,
	\begin{align*}
	W\subseteq Z(C_1+C_2)&\iff C_1+C_2\subseteq\Ann(W) \\ 
	& \iff C_1,C_2\subseteq\Ann(W)\iff W\subseteq Z(C_1)\cap Z(C_2);
	\end{align*}
	taking $W$ to be each of the two root spaces $Z(C_1+C_2)$ and
	$Z(C_1)\cap Z(C_2)$ gives the first identity, and the third is
	proved symmetrically. The remaining two follow from these by
	Theorem~\ref{thm:lattice}: applying $\Ann$ to the first identity
	with $C_i=\Ann(W_i)$, and using $Z(\Ann(W_i))=W_i$ and
	$\Ann(Z(C))=C$, yields
	$\Ann(W_1\cap W_2)=\Ann(W_1)+\Ann(W_2)$; similarly for
	$Z(C_1\cap C_2)$.
	
	(2) A map $f\in R=\End_{R_0}(M)$ vanishes on $W$ if and only if it
	factors through $M/W$, so
	\[\Ann(W)\cong\Hom_{R_0}(M/W,M)\cong\Hom_{R_0}(M/W,R_0)^{\oplus m}.\]
	For any finite $R_0$-module $N$, the Gorenstein isomorphism and
	the adjunction
	\[
	\Hom_{R_0}\bigl(N,\Hom_\kappa(R_0,\kappa)\bigr)\cong
	\Hom_\kappa(N,\kappa)
	\]
	give $\dim_\kappa\Hom_{R_0}(N,R_0)=\dim_\kappa N$. Taking
	$N=M/W$, of $\kappa$-dimension $n-\dim_\kappa W$, yields
	$\dim_\kappa\Ann(W)=m\,(n-\dim_\kappa W)$.
	
	Now let $\gcd(m,d)=1$ and $C=Rg$. Then $Z(C)=Z(g)$, and
	$C=\Ann(Z(C))$ by Theorem~\ref{thm:lattice}. Since $K\subseteq R$,
	the left ideal $C$ is a $K$-subspace, so
	$\dim_\kappa C=m\dim_KC=m\,(n-\deg_yg)$ by
	Theorem~\ref{thm:generator}(2). Comparing with
	$\dim_\kappa C=m\,(n-\dim_\kappa Z(C))$
	gives $\dim_\kappa Z(C)=\deg_yg=n-\dim_KC$.
\end{proof}

\section{The rank metric and the bounds}\label{s:bounds}

Throughout this section, the standing assumptions are those fixed in
Section~\ref{s:module}: $\phi_T=t+\tau^n$ is supersingular, so
$\gcd(n,r)=1$, the prime $\fp\neq\fp_0$ of degree $d$ is arbitrary, and
$m\mid n$ is arbitrary, with $s=n/m$. We emphasize that $\gcd(m,d)=1$ is
\emph{not} assumed: the results below rest on the splitting of
Proposition~\ref{prop:dc}(4) and the duality of
Corollary~\ref{cor:lattice}, not on the presentation of
Theorem~\ref{thm:presentation}.

\subsection{The metric as extra structure}\label{ss:metric}
Sections~\ref{s:codes}--\ref{s:roots} used only the pair $(R,R_0)$
acting on $M$. The rank metric is \emph{not} intrinsic to this data:
it depends on the splitting of Proposition~\ref{prop:dc}(4). Fix an
$\F_q$-basis $(s_j)_{1\leq j\leq nd}$ of $S$; by bijectivity of
$L\otimes_{\F_q}S\to R$, every $f\in R$ has a unique expansion
$f=\sum_jb_js_j$ with $b_j\in L$. Define the \emph{support}
$V(f)=\langle b_1,\dots,b_{nd}\rangle_{\F_q}\subseteq L$. Now put 
\begin{equation}\label{eq:wtmin}
\wtR(f)=\dim_{\F_q}V(f)
=\min\{\dim_{\F_q}V:\ V\subseteq L,\ f\in V\cdot S\},
\end{equation}
so that $V(f)$, and hence the weight, do not depend on the chosen
basis.

\begin{proof}[Proof of the second equality in \eqref{eq:wtmin}]
If $f\in V\cdot S$, write $f=\sum_kv_kt_k$ with $v_k\in V$,
$t_k\in S$, and expand each $t_k$ in the basis $(s_j)$. Comparing the
$L$-coefficients of $s_j$ in the two expansions of $f$ shows
$b_j\in V$ for all $j$, i.e., $V\supseteq V(f)$. Conversely,
$f\in V(f)\cdot S$ by definition. Hence the minimum is attained at
$V=V(f)$.
\end{proof}

\begin{defn}
For a code $C$ we put \[\dR(C)=\min\{\wtR(f):0\neq f\in C\}.\] Note
$\wtR(f)\leq m$ always.
\end{defn}

\subsection{Restricted Moore nondegeneracy}\label{ss:moore}

Full nonsingularity of $q^r$-Moore matrices of $\F_q$-independent
elements is \emph{false} in this setting
(Remark~\ref{rem:moorefalse}). What we prove in this subsection 
is a ``restricted" Moore nondegeneracy (Theorem~\ref{thm:moorefamily}), which is sufficient for the proof of BCH-type bounds 
for the metric introduced above. 

We need two preliminary lemmas. 

\begin{lem}\label{lem:kernelrat}
Let $\gamma_1,\dots,\gamma_w\in\overline\F_q$ and let
\[
K_0=\Bigl\{\vec{c}\in\F_{q^r}^{\,w}:\ \sum_{k=1}^w c_k\gamma_k=0\Bigr\}
\]
be the space of $\F_{q^r}$-rational relations. The solution space $V$ in
$\overline\F_q^{\,w}$ of the system of linear equations 
\[
\sum_{k=1}^{w}x_k\gamma_k^{[ri]}=0,\qquad 0\leq i\leq w-1,
\]
is $\overline\F_q\otimes_{\F_{q^r}}K_0$.
\end{lem}

\begin{proof}
	The solution space $V$ is an $\oF_q$-subspace of $\oF_q^{\,w}$, so
	for the inclusion $\oF_q\otimes_{\F_{q^r}}K_0\subseteq V$ it
	suffices to show $K_0\subseteq V$. For
	$\vec{c}=(c_1,\dots,c_w)\in K_0$ one has $c_k^{[ri]}=c_k$ for all
	$k$ and all $i\geq0$, so, for $0\leq i\leq w-1$,
	\[
	\sum_{k=1}^{w}c_k\gamma_k^{[ri]}
	=\left(\sum_{k=1}^{w}c_k\gamma_k\right)^{[ri]}=0 .
	\]
	
	To prove the reverse inclusion, let
	$e=\dim_{\F_{q^r}}\langle\gamma_1,\dots,\gamma_w\rangle_{\F_{q^r}}$,
	choose an $\F_{q^r}$-basis $\eta_1,\dots,\eta_e$ of the span, and
	write $\gamma_k=\sum_{j=1}^{e}c_{jk}\eta_j$, $1\leq k\leq w$, with
	$C=(c_{jk})\in\Mat_{e\times w}(\F_{q^r})$. First observe that
	$K_0=\ker_{\F_{q^r}}(C)$: for $\vec{c}\in\F_{q^r}^{\,w}$,
	\[
	\sum_{k=1}^{w}c_k\gamma_k
	=\sum_{j=1}^{e}\left(\sum_{k=1}^{w}c_{jk}c_k\right)\eta_j
	=\sum_{j=1}^{e}(C\vec{c})_j\,\eta_j ,
	\]
	which vanishes if and only if $C\vec{c}=0$, by independence of the
	$\eta_j$. Consequently
	\[
	\ker_{\oF_q}(C)=\oF_q\otimes_{\F_{q^r}}K_0. 
	\]
	
	Since the entries of $C$ are fixed by $\tau^r$, the coefficient
	matrix of the system factors as
	\[
	\bigl(\gamma_k^{[ri]}\bigr)_{i, k}
	=\bigl(\eta_j^{[ri]}\bigr)_{i, j}
	\cdot C .
	\]
	Now let $\vec{b}\in V$. Then
	$\bigl(\eta_j^{[ri]}\bigr)_{i, j}$ kills
	the vector $C\vec{b}\in\oF_q^{\,e}$. The top $e\times e$ block of
	this matrix (note $w\geq e$) is the $q^r$-Moore matrix of the
	$\F_{q^r}$-independent elements $\eta_1,\dots,\eta_e$, nonsingular
	by Moore's theorem \cite[Thm.~3.1.18]{PapikianBook}, applied with
	$q^r$ in place of $q$. Thus $C\vec{b}=0$, that is,
	$\vec{b}\in\ker_{\oF_q}(C)=\oF_q\otimes_{\F_{q^r}}K_0$. This proves
	$V\subseteq\oF_q\otimes_{\F_{q^r}}K_0$.
\end{proof}

\begin{lem}\label{lem:pushdown}
Assume $\gcd(m,r)=1$, and let $K_0\subseteq\F_{q^r}^{\,w}$ be an
$\F_{q^r}$-subspace. If
\[(\overline\F_q\otimes_{\F_{q^r}}K_0)\cap\F_{q^m}^{\,w}\neq0,\] then
$K_0\cap\F_q^{\,w}\neq0$.
\end{lem}

\begin{proof}
	Let $\vec{b}\neq0$ lie in the intersection and let
	\[
	U\colonequals\bigl\langle\,\rho(\vec{b}):\ \rho\in
	\Gal(\oF_q/\F_{q^r})\,\bigr\rangle_{\oF_q},
	\]
	the Galois group acting coordinatewise. Since
	$\oF_q\otimes_{\F_{q^r}}K_0$ has a basis with entries in $\F_{q^r}$, it
	is stable under $\Gal(\oF_q/\F_{q^r})$, so
	$U\subseteq\oF_q\otimes_{\F_{q^r}}K_0$. The coordinates of $\vec{b}$ lie
	in $\F_{q^m}$, and, because $\F_{q^m}\cap\F_{q^r}=\F_q$, the restriction
	map $\Gal(\oF_q/\F_{q^r})\to\Gal(\F_{q^m}/\F_q)$ is surjective; hence
	the set $\{\rho(\vec{b})\}$ is the full
	$\Gal(\F_{q^m}/\F_q)$-orbit of $\vec{b}$. It follows that $U$ is stable
	under all of $\Gal(\oF_q/\F_q)$: it is stable under
	$\Gal(\oF_q/\F_{q^r})$ by construction, and for
	$\theta\in\Gal(\oF_q/\F_{q^m})$ each $\theta\rho(\vec{b})$ depends
	only on $(\theta\rho)|_{\F_{q^m}}$, which is again in the image. By
	Galois descent, $U$ is spanned by its $\F_q$-rational vectors; being
	nonzero, it contains some $0\neq\vec{b}_0\in\F_q^{\,w}$. We claim
	$\vec{b}_0\in K_0$, which finishes the proof. Indeed,
	$\vec{b}_0\in U\subseteq\oF_q\otimes_{\F_{q^r}}K_0$ and
	$\vec{b}_0\in\F_{q^r}^{\,w}$. On the other hand, $(\oF_q\otimes_{\F_{q^r}}K_0) \cap \F_{q^r}^{\,w} =K_0$. 
\end{proof}

\begin{thm}\label{thm:moorefamily}
Let $\gcd(m,r)=1$, let $\gamma_1,\dots,\gamma_w\in\overline\F_q$ be
linearly independent over $\F_q$, and suppose $\vec{b}\in\F_{q^m}^{\,w}$
satisfies $\sum_{k=1}^wb_k\gamma_k^{[ri]}=0$ for $0\leq i\leq w-1$. Then
$\vec{b}=0$.
\end{thm}

\begin{proof}
By Lemma~\ref{lem:kernelrat},
$\vec{b}\in\overline\F_q\otimes_{\F_{q^r}}K_0$. If $\vec{b}\neq0$,
Lemma~\ref{lem:pushdown} produces $0\neq \vec{c}\in K_0\cap\F_q^{\,w}$,
i.e., an $\F_q$-linear relation $\sum_{k=1}^wc_k\gamma_k=0$, contradicting
the independence of the $\gamma_k$. 
\end{proof}

\begin{rem}\label{rem:moorefalse}
The restriction to $\vec{b}\in\F_{q^m}^{\,w}$ is essential: the square
$q^r$-Moore matrix of $\F_q$-independent elements can be singular.
Indeed, suppose $r\geq2$, $\fp_0-1$ is irreducible, and take $\fp=\fp_0-1$. Then
$M=\F_{q^{nr}}$ (see Remark~\ref{rem1.6}). Elements of
$M$ that are independent over $\F_q$ are far from independent over
$\F_{q^r}$ (indeed, $\dim_{\F_{q^r}}M=n<nr=\dim_{\F_q}M$), and any $\F_{q^r}$-relation among them makes the full
Moore matrix singular by Lemma~\ref{lem:kernelrat}. 
\end{rem}

\subsection{The rank-BCH bound}\label{ss:bch}
	Let $\alpha\in M=\phi[\fp]$ and let $S=\kappa[\pi|_M]$. We say that $\alpha$ is \textit{$S$-normal} if the
	following equivalent conditions hold:
	\begin{enumerate}
		\item $\alpha,\pi(\alpha),\dots,\pi^{n-1}(\alpha)$, i.e., 
		$\alpha,\alpha^{[r]},\dots,\alpha^{[r(n-1)]}$, is a $\kappa$-basis of
		$M$;
		\item $\alpha$ is a cyclic vector for $\pi|_M$ (cf. Proposition~\ref{prop:free});
		\item the evaluation map $\ev_\alpha\colon S\To M$,
		$s\mapsto s(\alpha)$, is bijective, i.e., $M$ is free of rank one
		over $S$ with generator $\alpha$;
		\item no nonzero element of $S$ annihilates $\alpha$.
	\end{enumerate}

\begin{thm}\label{thm:rankbch}
Let $\alpha\in M$ be $S$-normal and let $0\neq f\in R$. If
\[
Z(f)\supseteq\{\alpha^{[ri]}:\ 0\leq i\leq\delta-2\},
\]
then $\wtR(f)\geq\delta$.
\end{thm}

\begin{proof}
Set $w\colonequals\wtR(f)$ and suppose $w\leq\delta-1$. Choose a
basis $b_1,\dots,b_w$ of the support $V(f)\subseteq L$. Expanding
$f=\sum_{j=1}^{nd}c_js_j$ over an $\F_q$-basis $(s_j)$ of $S$ as in
Section~\ref{ss:metric}, each coefficient $c_j\in V(f)$ has a unique
expansion
\[
c_j=\sum_{k=1}^{w}a_{kj}b_k,\qquad a_{kj}\in\F_q ,
\]
and regrouping gives
\[
f=\sum_{k=1}^{w}b_kt_k,\qquad
t_k\colonequals\sum_{j=1}^{nd}a_{kj}s_j\in S .
\]
Let $A=(a_{kj})\in\Mat_{w\times nd}(\F_q)$. The columns of $A$ are
the coordinate vectors of $c_1,\dots,c_{nd}$ in the basis $(b_k)$. 
Since the $c_j$ span the $w$-dimensional space $V(f)$, the column
rank of $A$ is $w$, hence so is the row rank. The $k$-th row of $A$
is the coefficient vector of $t_k$ in the basis $(s_j)$ of $S$, so
$t_1,\dots,t_w$ are $\F_q$-linearly independent in $S$.

Since $S$ is commutative, we have $t_k\pi^i=\pi^it_k$. 
Setting $\gamma_k\colonequals t_k(\alpha)$, for
$0\leq i\leq\delta-2$ the hypothesis gives 
\[
0=f\bigl(\alpha^{[ri]}\bigr)
=\sum_{k=1}^{w}b_k\,t_k\bigl(\pi^i(\alpha)\bigr)
=\sum_{k=1}^{w}b_k\,\gamma_k^{[ri]} ,
\]
a genuine $\overline\F_q$-linear system, because the $b_k\in L$ act
as scalars. By normality, the evaluation $\ev_\alpha$ is injective, so $\gamma_1,\dots,\gamma_w$ are
$\F_q$-linearly independent, and $\vec{b}\in\F_{q^m}^{\,w}$ satisfies the
first $w$ of the displayed equations (available since
$\delta-1\geq w$). Theorem~\ref{thm:moorefamily}, applicable because
$m\mid n$ and $\gcd(n,r)=1$ imply $\gcd(m,r)=1$, forces
$b_1=\dots=b_w=0$, i.e., $f=0$, a contradiction.
\end{proof}

\begin{defn}
	\label{def:rankbchcode} Let $\alpha\in M$. 
	By analogy with \cite{MP}, we call the cyclic submodule $R_0\alpha$ the \textit{cyclotomic
		space} of $\alpha$; this is the 
		smallest root space containing $\alpha$. 
	For $S$-normal  
	$\alpha$ and $1\leq\delta\leq m$, the
	\emph{rank-BCH code of designed distance $\delta$} is
	$C=\Ann(W_\delta)$, where
	\[
	W_\delta=\bigoplus_{i=0}^{\delta-2}\pi^iR_0\,\alpha
	=\bigoplus_{i=0}^{\delta-2}R_0\,\alpha^{[ri]}
	\]
	is the sum of the cyclotomic spaces of the first $\delta-1$ points
	of the $\pi$-orbit of $\alpha$. The space $W_\delta$ is the smallest root space containing
	$\{\alpha^{[ri]}:0\leq i\leq\delta-2\}$, so that $C$ is the largest
	code vanishing on that orbit segment. The sum is direct inside
	$M=\bigoplus_{i<m}\pi^iR_0\,\alpha$, so 
	\begin{equation}\label{eq:dimWdelta}
		\dim_\kappa W_\delta=s(\delta-1).\end{equation}
\end{defn}

\begin{prop}[Singleton bound]\label{prop:singleton}
Let $C\neq0$ be a code and $k\colonequals\dim_LC$ $($every left ideal
is an $L$-subspace$)$. Let, as before, $d=\deg\fp$. Then
\[
\dR(C)\ \leq\ \Bigl\lfloor\frac{m\,(nd-k)}{nd}\Bigr\rfloor+1 .
\]
\end{prop}

\begin{proof}
For an $\F_q$-subspace $V\subseteq L$ of dimension $v$, bijectivity
of the splitting $L\otimes_{\F_q} S \to R$
gives \[\dim_{\F_q}(V\cdot S)=v\cdot \dim_{\F_q}S = v\cdot nd\] exactly. If
$mk+v\cdot nd>\dim_{\F_q}R=mnd$, i.e.,  $v>m(nd-k)/nd$, then
$C\cap (V\cdot S)\neq0$, producing a nonzero element of weight
$\leq v$. Take $v=\lfloor m(nd-k)/nd\rfloor+1$; note $v\leq m$, because $k\geq1$.
\end{proof}

\begin{thm}\label{thm:exact}
Let $\alpha\in M$ be $S$-normal, $1\leq\delta\leq m$, and let $C=\Ann(W_\delta)$ be the
rank-BCH code of designed distance $\delta$. Then
\begin{align*}
\dim_{\F_q}C &=nd\,(m-\delta+1),\\ \dR(C) &=\delta. 
\end{align*}
Thus, the designed distance is the true distance.
\end{thm}

\begin{proof} By definition, $\dim_\kappa W_\delta=s(\delta-1)$; cf. \eqref{eq:dimWdelta}. 
Plugging this into Corollary~\ref{cor:lattice}(2) gives 
\[
\dim_\kappa C=m\bigl(n-s(\delta-1)\bigr)=ms\,(m-\delta+1)
=n\,(m-\delta+1),
\]
i.e., $\dim_{\F_q}C=nd(m-\delta+1)$. 

Every nonzero $f\in C$ vanishes on
\[W_\delta\supseteq\{\alpha^{[ri]}:i\leq\delta-2\},\] so
$\dR(C)\geq\delta$ by Theorem~\ref{thm:rankbch}. 
For the reverse
inequality, note that \[\dim_LC=\frac{\dim_{\F_q}C}{m}=sd\,(m-\delta+1)\] and 
\[
\Bigl\lfloor\frac{m\,(nd-sd(m-\delta+1))}{nd}\Bigr\rfloor+1
=\Bigl\lfloor\frac{m\,sd(\delta-1)}{msd}\Bigr\rfloor+1=\delta, 
\]
so $\dR(C)\leq\delta$ by Proposition~\ref{prop:singleton}. 
\end{proof}

\subsection{An axiomatic formulation}\label{ss:axioms}

The proofs of Sections~\ref{s:roots}--\ref{s:bounds} are based on a short
list of properties of the pair $(R,R_0)$. We record the list, for a supersingular $\phi$ of
rank $n$ over a finite $A$-field (all endomorphisms defined over the
base), a prime $\fp$ away from the characteristic,
$\kappa=A/\fp$, $M=\phi[\fp]$, and
$E=\End_\kappa(M)\cong\Mat_n(\kappa)$:

\begin{enumerate}
	\item[(A1)] (\textbf{Structure}) A commutative subring
	$\kappa\subseteq R_0\subseteq E$, Gorenstein as a $\kappa$-algebra,
	such that $M$ is free over $R_0$. Put
	\[R\colonequals\Cent_E(R_0)=\End_{R_0}(M).\] \emph{Codes} are left
	ideals of $R$, \emph{root spaces} are $R_0$-submodules of $M$.
	
	\item[(A2)] (\textbf{Metric}) A subfield $\F_q\subseteq L\subseteq
	R$ and an element $\sigma\in R$ such that, with
	$S\colonequals\kappa[\sigma]\supseteq R_0$, multiplication
	$L\otimes_{\F_q}S\To R$ is bijective; the weight of $f\in R$ is the
	minimal $\F_q$-dimension of a $V\subseteq L$ with $f\in V\cdot S$.
	
	\item[(A3)] (\textbf{Nondegeneracy}) The element $\sigma$ is non-derogatory on $M$
	over $\kappa$, and for all $\F_q$-independent
	$\gamma_1,\dots,\gamma_w\in M$ the map
	\begin{align*} L^w &\To M^w,\\ 
		(b_k) &\longmapsto\bigl(\sum_{k=1}^w b_k\sigma^i(\gamma_k)\bigr)_{0\leq i\leq w-1},
	\end{align*}
	is injective.
\end{enumerate}

Note that (A3) is an assertion about the action of $L$ through $E$;
no assumption is made that $L$ consists of scalar multiplications.
In the family of this paper it does, and that is what converts the
map in (A3) into a Moore-type system. 

The lattice anti-isomorphism (Theorem~\ref{thm:lattice}) holds under (A1) alone. 
Given (A1)--(A3), the proofs of Theorem~\ref{thm:rankbch} and
Proposition~\ref{prop:singleton} go through verbatim: nothing in them
uses more than the axioms. (The exact-distance theorem,
Theorem~\ref{thm:exact}, uses in addition the grading
$S=\bigoplus_{i<m}\sigma^iR_0$, which in this paper comes from
$R_0=\kappa[\sigma^m]$ and is not part of the axioms.)
On the other hand, principality of the code
ideals is \emph{not} axiomatic: it came from $R_0$ being a quotient
of a principal ideal domain (Proposition~\ref{prop:principal}). 

This paper realizes the axioms on the family
$\phi_T=t+\tau^n$, for every $m\mid n$ and every prime
$\fp\neq\fp_0$. Whether the axioms admit instances beyond this specific family 
will be the subject of a future investigation.


\section{Matrix realization, MRD codes, and prescribed parameters}\label{s:matrix}

Fix an $\F_q$-basis $\lambda_0,\dots,\lambda_{m-1}$ of $L=\F_{q^m}$ and an
$\F_q$-basis $s_1,\dots,s_{nd}$ of $S$. By Proposition~\ref{prop:dc}(4),
every $f\in R$ has a unique expansion
\[
f=\sum_{j=1}^{nd}b_js_j,\qquad b_j\in L. 
\]
Writing $b_j=\sum_{i=0}^{m-1}a_{ij}\lambda_i$ with $a_{ij}\in\F_q$,  we
put
\[
A_f\colonequals(a_{ij})\in\Mat_{m\times nd}(\F_q),
\]
so that the $j$-th column of $A_f$ is the coordinate vector of $b_j$. 

For a
code $C$,  put $A(C)=\{A_f: f\in C\}$. For $\zeta\in L$, let
$M_\zeta\in\Mat_m(\F_q)$ be the matrix of multiplication by $\zeta$ on $L$
in the basis $(\lambda_i)$; the $M_\zeta$ form a subfield
$M_L\cong\F_{q^m}$ of $\Mat_m(\F_q)$.

\begin{lem}\label{lem:isometry} We have:
\begin{enumerate}
\item The map $f\mapsto A_f$ is $\F_q$-linear and injective, and
$\operatorname{rank}(A_f)=\wtR(f)$ for every $f\in R$.
\item $A_{\zeta f}=M_\zeta A_f$ for $\zeta\in L$ and $f\in R$.
\end{enumerate}
Consequently, for every code $C$, the set $A(C)$ is an $\F_q$-linear
rank-metric code in $\Mat_{m\times nd}(\F_q)$ with
$\dim_{\F_q}A(C)=\dim_{\F_q}C$ and $\dR(A(C))=\dR(C)$, closed under left
multiplication by the field $M_L$.
\end{lem}

\begin{proof}
(1) Linearity and injectivity follow from the uniqueness of the expansion.
The column space of $A_f$ is the space of coordinate vectors of the
$\F_q$-span of $b_1,\dots,b_{nd}$, which is the support $V(f)$ of
Section~\ref{ss:metric}. Hence $\rank(A_f)=\dim_{\F_q}V(f)=\wtR(f)$
by \eqref{eq:wtmin}. 

(2) Since $\zeta f=\sum_j(\zeta b_j)s_j$, the $j$-th
column of $A_{\zeta f}$ is the coordinate vector of $\zeta b_j$, which is
$M_\zeta$ applied to the coordinate vector of $b_j$.
\end{proof}

Recall the rank-metric \textit{Singleton bound} \cite{Delsarte,Gabidulin}: if
$\mathcal C\subseteq\Mat_{a\times b}(\F_q)$ with $a\leq b$ is an
$\F_q$-linear code of minimum rank distance $\delta$, then
\[
\dim_{\F_q}\mathcal C\ \leq\ b\,(a-\delta+1),
\]
and $\mathcal C$ is called \emph{MRD} if equality holds.

\begin{thm}\label{thm:mrd}
Let $m\mid n$ and let $\fp\neq\fp_0$ be arbitrary $($no coprimality between
$m$ and $d$ is assumed$)$. Let $\alpha\in M$ be $S$-normal, $1\leq\delta\leq m$,
and let $C=\Ann(W_\delta)$ be the rank-BCH code of designed distance
$\delta$. Then $A(C)\subseteq\Mat_{m\times nd}(\F_q)$ is an
$\F_{q^m}$-linear MRD code:
\[
\dim_{\F_q}A(C)=nd\,(m-\delta+1),\qquad \dR(A(C))=\delta .
\]
\end{thm}

\begin{proof}
By Theorem~\ref{thm:exact} and Lemma~\ref{lem:isometry},
$\dim_{\F_q}A(C)=nd(m-\delta+1)$ and $\dR(A(C))=\delta$. Since
$m\leq n\leq nd$, the Singleton bound for $\Mat_{m\times nd}(\F_q)$ reads
\[\dim_{\F_q}A(C)\leq nd(m-\delta+1),\] and it is attained. The
$\F_{q^m}$-linearity is Lemma~\ref{lem:isometry}(2).
\end{proof}

\begin{rem}\label{rem:mrd}
 Take $d=1$ and $m=n$. Then $\kappa=\F_q$, $s=1$, $R_0=\F_q$,
$R=E\cong\Mat_n(\F_q)$, $S=\F_q[\pi|_M]$ with basis $1,\pi,\dots,\pi^{n-1}$,
and $K=L=\F_{q^n}$. An element $f=\sum_{i<n}f_i\pi^i$, $f_i\in L$, acts on
$M$ by $v\mapsto\sum_if_iv^{[ri]}$, and $A_f$ is the coordinate matrix of
the vector $(f_0,\dots,f_{n-1})\in L^n$, whose rank weight is the usual
one. Root spaces are $\F_q$-subspaces of $M$, and for $S$-normal $\alpha$
the code $C=\Ann(W_\delta)$ consists of the $f$ with
\[
\sum_{j=0}^{n-1}f_j\,\alpha^{[r(i+j)]}=0,\qquad 0\leq i\leq\delta-2 .
\]
Since $M=L\alpha$, set
\[
\beta_j=\alpha^{[rj]}/\alpha\in L,\qquad 0\leq j<n .
\]
The elements $\beta_0,\dots,\beta_{n-1}$ are $\F_q$-linearly
independent, and $\alpha^{[r(i+j)]}=\beta_j^{[ri]}\alpha^{[ri]}$, so
dividing the $i$-th equation above by $\alpha^{[ri]}$ shows that, as a
code in $L^n$, $C$ has the $q^r$-Moore matrix
$\bigl(\beta_j^{[ri]}\bigr)_{0\leq i\leq\delta-2,\,0\leq j\leq n-1}$
as parity-check matrix. Hence $C$ is the generalized Gabidulin code of
\cite{KG} of dimension $n-\delta+1$ with Frobenius parameter $r$, in the
notation of \cite[\S2.2]{MP} the code
$\mathrm{Gab}_{n-\delta+1,r}(\beta_0,\dots,\beta_{n-1})$, and at $r=1$
a classical Gabidulin code \cite{Gabidulin}. (When $c=1$ one has
$M=L$, and one may use $\alpha^{[rj]}$ in place of $\beta_j$, which is
the normal-basis description of \cite[Thm.~6]{MP}.)
\end{rem}

We now choose the bases so that multiplication by $\pi$ becomes visible on
the matrices. Since $\gcd(r,m)=1$, an element $\theta\in L$ is normal over
$\F_q$ if and only if $\lambda_i\colonequals\theta^{[ri]}$, $0\leq i<m$, is
a basis of $L$; fix such a $\theta$. Fix $\varepsilon\in\kappa$ normal over
$\F_q$ and put $\varepsilon_\ell=\varepsilon^{[\ell]}$, $0\leq\ell<d$. By
Proposition~\ref{prop:free}(1), the elements
\[
s_{k,\ell}\colonequals\varepsilon_\ell\,\pi^k,\qquad 0\leq k<n,\ 0\leq\ell<d,
\]
form an $\F_q$-basis of $S$, and we index the columns of $A_f$ by the pairs
$(k,\ell)$. Let $Q\in\Mat_m(\F_q)$ be the permutation matrix of the cyclic
shift $i\mapsto i+1$ of $\{0,\dots,m-1\}$, and let $T\in\Mat_{nd}(\F_q)$ be
the matrix with
\[
T_{(k,\ell),(k+1,\ell)}=1\quad(0\leq k\leq n-2),\qquad
T_{(n-1,\ell),(0,\ell')}=c_{\ell\ell'},
\]
where $c\,\varepsilon_\ell=\sum_{\ell'}c_{\ell\ell'}\varepsilon_{\ell'}$ in
$\kappa$, and all other entries zero. Thus right multiplication by $T$
shifts the $n$ blocks of $d$ columns by one, and feeds the last block back
into the first through the matrix of multiplication by $c$ on $\kappa$.

\begin{prop}\label{prop:shift}
For every $f\in R$,
\[
A_{\pi f}=Q\,A_f\,T .
\]
Both $Q$ and $T$ are invertible. Consequently, for every code $C$, the
matrix code $A(C)$ is invariant under $A\mapsto QAT$ and under
$A\mapsto M_\zeta A$, $\zeta\in L$.
\end{prop}

\begin{proof}
Write $f=\sum_{k,\ell}b_{k,\ell}s_{k,\ell}$. Since $\pi\zeta=\zeta^{[r]}\pi$
for $\zeta\in L$ and $\pi$ commutes with $\kappa$ ($S$ is commutative),
\[
\pi f=\sum_{k,\ell}b_{k,\ell}^{[r]}\,\varepsilon_\ell\,\pi^{k+1},
\]
where $\varepsilon_\ell\pi^{k+1}=s_{k+1,\ell}$ for $k\leq n-2$, while for
$k=n-1$ we have $\varepsilon_\ell\pi^n=c\,\varepsilon_\ell=\sum_{\ell'}c_{\ell\ell'}s_{0,\ell'}$.
If $b=\sum_ia_i\lambda_i$, then
$b^{[r]}=\sum_ia_i\theta^{[r(i+1)]}=\sum_ia_i\lambda_{i+1}$ (indices modulo
$m$, as $\theta^{[rm]}=\theta$), so the coordinate vector of $b^{[r]}$ is
$Q$ applied to that of $b$. Reading off columns: column $(k+1,\ell)$ of
$A_{\pi f}$ is $Q$ times column $(k,\ell)$ of $A_f$ for $k\leq n-2$, and
column $(0,\ell')$ of $A_{\pi f}$ is $\sum_\ell c_{\ell\ell'}\,Q$ times column
$(n-1,\ell)$ of $A_f$; this is $A_{\pi f}=QA_fT$. The matrix $Q$ is a
permutation matrix, and $T$ is invertible because $c\neq0$. The invariance
statements follow, as $C$ is a left ideal.
\end{proof}

\begin{rem}
	MRD codes exist in $\Mat_{a\times b}(\F_q)$ for all $a\leq b$ and all
	$\delta$ (Delsarte's codes), so
	for $s>1$ or $d>1$ the content of Theorem~\ref{thm:mrd} lies in the
	structure of the codes rather than in their parameters: they are
	$\F_{q^m}$-linear, they are left ideals of $\Mat_m(R_0)$, and they carry
	the symmetry of Proposition~\ref{prop:shift}. Whether the codes
	$A(C)$ with $s>1$ or $d>1$ are equivalent to previously known MRD codes is
	a question we leave open.
\end{rem}

The data entering Theorems \ref{thm:exact} and \ref{thm:mrd} are $q$, the
integers $m$ and $s$ (so $n=ms$), the prime $\fp$ of degree $d$, and the
constant $c\in\kappa^\times$. The Drinfeld module itself enters only through
$c=\fp_0\bmod\fp$ and, when $\gcd(m,d)=1$, through the class of $r=\deg\fp_0$
modulo $m$, which determines the automorphism $\sigma$ of
Theorem~\ref{thm:presentation}. The construction can therefore be run
backwards from prescribed data, by Dirichlet's theorem for $A$; the same
tool was used in \cite{BDM24} to construct rank-metric codes with
rank-locality from Drinfeld modules.

\begin{prop}\label{prop:dirichlet}
Let $m,s\geq1$, $n=ms$, let $\fp\in A$ be monic irreducible of degree $d$,
let $c\in\kappa^\times$, and let $e_0$ be an integer with $\gcd(e_0,m)=1$.
There exist infinitely many monic irreducible $\fp_0\in A$ with
\[
\fp_0\equiv c\pmod{\fp},\qquad \gcd(\deg\fp_0,\,n)=1,\qquad
\deg\fp_0\equiv e_0\pmod{m}.
\]
For each such $\fp_0$ and each root $t$ of $\fp_0$, the module
$\phi_T=t+\tau^n$ is supersingular, $\fp\neq\fp_0$, and its $\fp$-torsion
realizes $R_0\cong\kappa[u]/(u^s-c)$, $R\cong\Mat_m(R_0)$, and the codes of
Sections~\ref{s:codes}--\ref{s:matrix} with these parameters. If
$\gcd(m,d)=1$, the automorphism $\sigma$ of Theorem~\textup{\ref{thm:presentation}}
is $\Frob^e|_K$ with $e\equiv e_0\pmod m$, $e\equiv0\pmod d$.
\end{prop}

\begin{proof}
The class of $c$ is invertible modulo $\fp$, so by Dirichlet's theorem for
$A$ \cite[Thm.~4.8]{Rosen} the number of monic irreducibles of degree $N$
congruent to $c$ modulo $\fp$ is \[q^N/(N(q^d-1))+O(q^{N/2}).\]
In particular, there is one in every sufficiently large
degree $N$. The congruence conditions $N\equiv e_0\pmod m$ and
$N\equiv1\pmod\ell$ for every prime $\ell\mid n$ with $\ell\nmid m$ have
pairwise coprime moduli, so they are simultaneously satisfied by infinitely
many $N$, and every such $N$ is coprime to $n$ (its prime factors dividing
$m$ are excluded by $\gcd(e_0,m)=1$). Supersingularity is
Lemma~\ref{lem:height}; $\fp\neq\fp_0$ because $c\neq0$; and
$c=\fp_0\bmod\fp$ by construction. The last statement is
Theorem~\ref{thm:presentation}(2).
\end{proof}

\section{Comparison with \cite{MP}}\label{sComparison}

	Suppose $\fp_0-1$ is irreducible and take $\fp=\fp_0-1$, so $d=r$,
	$c=1$, and $M=\F_{q^{nr}}$ with $\pi|_M$ the $q^r$-power Frobenius. 
	Note that $\gcd(m,d)=\gcd(m,r)=1$ automatically.
	
	\begin{lem}\label{lem:mp}
		Let $R_{\mathrm{MP}}$ denote the ring of $q^r$-linearized
		polynomials over $\F_{q^m}$ modulo $x^{[rn]}-x$, acting on
		$\F_{q^{rn}}$ by evaluation, whose left ideals are the skew cyclic
		codes of \cite{MP}. Then the subring of $R$ generated by $L$ and
		$\pi$ coincides with $R_{\mathrm{MP}}$ as a ring of operators on
		$M$, and multiplication $\kappa\otimes_{\F_q}R_{\mathrm{MP}}\To R$
		is an isomorphism of rings. In particular, $R=R_{\mathrm{MP}}$ if
		and only if $r=1$, in which case the codes of this paper at
		$\fp=\fp_0-1$ are exactly the skew cyclic codes of \cite{MP}, with
		$S$-normal elements corresponding to normal bases of
		$\F_{q^n}/\F_q$.
	\end{lem}

	\begin{proof}
		The constants $L$ act on $M$ by scalar multiplication and $\pi$ by
		$v\mapsto v^{[r]}$, so $\sum_{i<n}F_i\pi^i$ with $F_i\in L$ is
		the operator $v\mapsto\sum_{i<n}F_iv^{[ri]}$. The assignment
		$x^{[ri]}\mapsto\pi^i$ therefore maps $R_{\mathrm{MP}}$ onto the
		subring $L[\pi]$ of $R$: it respects the commutation rule of
		$q^r$-polynomials, since $\pi F=F^{[r]}\pi$ for $F\in L$, and it
		kills $x^{[rn]}-x$, since $\pi^n=1$ on $M$. It is injective,
		because the maps $v\mapsto v^{[ri]}$, $0\leq i<n$, are distinct
		automorphisms of $\F_{q^{rn}}$, hence linearly independent over
		$\F_{q^{rn}}\supseteq L$ by Artin's theorem. Thus
		$\dim_{\F_q}L[\pi]=mn$.

		Since $\kappa\subseteq R_0$ and $R$ centralizes $R_0$, the
		elements of $\kappa$ commute with $R$, so multiplication
		$\kappa\otimes_{\F_q}L[\pi]\To R$ is a ring homomorphism. It is
		surjective because $R=L\cdot S=L\cdot\kappa[\pi]=\kappa\cdot L[\pi]$
		by Proposition~\ref{prop:dc}(4), and
		$\dim_{\F_q}(\kappa\otimes_{\F_q}L[\pi])=r\cdot mn=mnd=\dim_{\F_q}R$,
		as $d=r$. Hence it is an isomorphism, and $R=L[\pi]$ if and only
		if $r=1$. In that case $\kappa=\F_q$ acts on $M=\F_{q^n}$ by
		scalars, $\pi$ is the $q$-power Frobenius, and $S$-normality of
		$\alpha$ means that $\alpha,\alpha^{[1]},\dots,\alpha^{[n-1]}$
		is an $\F_q$-basis of $\F_{q^n}$.
	\end{proof}
	
	In this comparison, the twist parameter $r$ of \cite{MP} is the degree of the $A$-characteristic $\fp_0$, 
	and the coprimality hypothesis $\gcd(r,n)=1$ under which
	\cite[Thm.~6]{MP} identifies rank-BCH codes with generalized Gabidulin
	codes (via \cite[Lem.~2]{KG} $=$ \cite[Lem.~4]{MP}) is, on our side,
	the supersingularity of $\phi$ (Lemma~\ref{lem:height}).  
	
	We also point out that 
	Corollaries 5 and 6 of \cite{MP} are \textbf{false} as stated. The hypothesis
	of \cite[Cor.~6]{MP} requires only that the \emph{first} $\delta-1$
	twists $\alpha,\alpha^{[r]},\dots,\alpha^{[r(\delta-2)]}$ be linearly
	independent over the constant field (for $\delta=2$, merely that
	$\alpha\neq0$), whereas \cite[Prop.~1]{ChaussadeLU} demands
	independence of the \emph{full} orbit, i.e., \textit{$S$-normality}. 
	For example, take $\fp_0=T$ (so $r=1$), $\fp=T-1$, any $m\geq2$ with $m\mid n$, and
	\[
	f=\pi-1,\qquad \alpha=1,\qquad \delta=2 .
	\]
	Then $f(1)=1^{[r]}-1=0$, so $\alpha\in Z(f)$ ($=\F_q$), and
	$\alpha\neq0$, so the hypotheses of \cite[Cor.~6]{MP} hold; yet the
	coefficients of $f$ lie in $\F_q$, so $\wtR(f)=1<2=\delta$.
	Proposition 7 of \cite{MP}, deduced from Corollary 6, fails for the
	same example: the rank-BCH code of designed distance $2$ attached to
	$\alpha=1$ is generated by $x^{[1]}-x$ and has minimum rank distance
	$1$. The gap is in the proof of \cite[Cor.~5]{MP}, where
	$\alpha^{[r(\delta-1)]}$ is taken to be linearly independent of the
	earlier twists, which the hypothesis does not provide.


\bibliographystyle{amsplain}
\bibliography{bibliography}

\end{document}